\documentclass[a4paper]{amsart}
 
\usepackage{amsthm}
\usepackage{amsmath}
\usepackage{amssymb}
\usepackage{latexsym}
\usepackage{enumerate}
\usepackage{graphicx}
\usepackage[utf8]{inputenc}
\usepackage{epsfig} 
\usepackage{pgf}
\usepackage{tikz}
\usepackage{float}
\usepackage{dsfont}
\usepackage{times}
\usepackage{enumerate}

\newcommand{\inj}{{{\rm inj}\,}}

\newcommand{\Minf}{{{\rm minfoc}\,}}

\newtheorem*{theorem*}{Theorem}

\renewcommand{\thetheoremName}

\newtheorem{theorem}{Theorem}[section]
\newtheorem{lemma}[theorem]{Lemma}
\newtheorem{proposition}[theorem]{Proposition}
\newtheorem{corollary}[theorem]{Corollary}

\theoremstyle{definition}
\newtheorem{definition}[theorem]{Definition}
\newtheorem{example}[theorem]{Example}

\newtheorem{remark}[theorem]{Remark}

\numberwithin{equation}{section}

\newcommand{\dist}{\operatorname{dist}}
\newcommand{\Vol}{\operatorname{Vol}}

\newcommand{\I}{\operatorname{inj}}

\newcommand{\erre}{\mathbb{R}}

\newcommand{\Inf}{\operatorname{Inf}}

\def\kam{I\!\!K^m(b)}

\newcommand{\Han}{\mathbb{H}^n(b)}
\newcommand{\Hatwo}{\mathbb{H}^2}
\newcommand{\Hathree}{\mathbb{H}^3}
\newcommand{\ene}{\mathbb{N}}

\begin{document}


\title[Mean curvature, volume and properness]{Mean curvature, volume and properness of isometric immersions}

\author[V. Gimeno]{Vicent Gimeno*}
\address{Departament de Matem\`{a}tiques- IMAC,
Universitat Jaume I, Castell\'{o}, Spain.}
\email{gimenov@uji.es}
\author[V. Palmer]{Vicente Palmer**}
\address{Departament de Matem\`{a}tiques- INIT,
Universitat Jaume I, Castellon, Spain.}
\email{palmer@mat.uji.es}
\thanks{* Work partially supported by the Research Program of University Jaume I Project P1-1B2012-18, and DGI -MINECO grant (FEDER) MTM2013-48371-C2-2-P}
\thanks{**Work partially supported by the Research Program of University Jaume I Project P1-1B2012-18,  DGI -MINECO grant (FEDER) MTM2013-48371-C2-2-P, and Generalitat Valenciana Grant PrometeoII/2014/064 }

\begin{abstract}
We explore the relation among volume, curvature and properness of a $m$-dimensional isometric immersion in a Riemannian manifold. We show that, when the $L^p$-norm of the mean curvature vector is bounded for some $m \leq p\leq \infty$, and the ambient manifold is a Riemannian manifold with bounded geometry, properness is equivalent to the finiteness of the volume of extrinsic balls. We also relate the total absolute curvature of a surface isometrically immersed in a Riemannian manifold with its properness. Finally, we relate the curvature and the topology of a complete and non-compact $2$-Riemannian manifold $M$ with non-positive Gaussian curvature and finite topology, using  the study of the focal points of the transverse Jacobi fields to a geodesic ray in $M$ . In particular, we have explored the relation between the minimal focal distance of a geodesic ray and the total curvature of an end containing that geodesic ray.

\end{abstract}

\maketitle
\section{Introduction}\label{Intr} 

In 1965, E. Calabi stated in \cite{Ca} the following conjecture, (also known as \lq\lq the first Calabi's conjecture"): {\em prove that a complete minimal hypersurface in $\erre^n$ must be unbounded}. This conjecture turned to be false when we consider {\em immersed} hypersurfaces: in the paper \cite{Na}, N. Nadirashvili constructed a complete immersion of a minimal disk into the unit ball in $\erre^3$. Nadirashvili's complete immersion is not proper, so it is natural to wonder about the r\^ole of properness in relation with the first Calabi's conjecture.

In fact, in the paper \cite{CoM1}, T. H. Colding and W. P. Minicozzi studied the following question (also known as the  Calabi-Yau conjecture): {\em Prove that every complete embedded minimal surface in $\erre^3$ is proper.}

In \cite{CoM1} the authors proved that a complete embedded minimal disk in $\erre^3$ must be proper, so the first Calabi's conjecture is true for embedded minimal disks. Moreover they proved that a complete embedded minimal surface with finite topology, (namely, homeomorphic to a compact $2$-dimensional manifold, from which it has been removed a finite number of points, which corresponds to the number of {\em ends} of the surface), is properly immersed in $\erre^3$.

On the other hand,  W. H. Meeks and H. Rosenberg proved in the paper \cite{Mer} that a complete embedded minimal surface with positive injectivity radius in $\erre^3$ must be proper.
 This result includes Colding's-Minicozzi result because they proved too in this paper that every complete embedded minimal surface with finite topology in $\erre^3$ has positive injectivity radius. We shall extract the intrinsic aspects of this proof (see Lemma \ref{MeeksRosenberg}) to show our Theorem \ref{Mainth1}.
 
More recently, Calabi-Yau's conjecture  has been explored in broader contexts, namely, W. H. Meeks, J. Perez and A. Ros have proved, (see \cite{MPR}),  that Calabi-Yau conjecture is true when we consider embedded minimal surfaces with finite genus and countable number of ends in $\erre^3$. On the other hand, B. Coskunuzer, W. H. Meeks and G. Tinaglia proved that Calabi-Yau conjecture is not true in $\Hathree$: there exists a non properly embedded minimal plane (genus zero, one end) in $\Hathree$, (see \cite{CMT}). In this way, M. Rodriguez and G. Tinaglia have constructed a family of non-proper complete minimal disks with finite total curvature in the quotient of $\Hatwo \times \erre$ by the discrete group of isometries generated by vertical translations and screw motions, which are embedded in $\Hatwo \times \erre$, (see the paper \cite{RT}).

Taking as a  starting point all these results, several questions arises in a natural way, e.g., what happens when we consider non minimal surfaces (or submanifolds) immersed in a more general ambient space?. Or, could the embeddedness of the immersion be replaced by another hypothesis related with the volume or the metric?. Summarizing: what can be said about geometric conditions describing the properness of an immersion in a broader context?

Looking for these geometric conditions and given an isometric immersion $\varphi : P^m \to N^n$, where $P^m$ and $N^n$ are complete Riemannian manifolds, a quite natural assumption to be considered is following
\begin{equation}\label{hypovol}
\Vol(\varphi^{-1}(B^N_t)) < \infty, \quad\forall t\, >\, 0\, .
\end{equation}
Where $B^N_t$ is the geodesic $t$-ball in the ambient manifold $N$. The domains $D_t=\varphi^{-1}(B^N_t)$ are known as the {\em extrinsic} $t$-balls. 

In fact, when the ambient manifold is complete, inequality (\ref{hypovol}) is a necessary condition for properness.  However, it is not a sufficient condition: consider e. g. a negatively curved and non-compact surface with finite volume which can be immersed (via Nash immersion theorem) in a metric ball in $\erre^N$ for large $N$.

On the other hand there are immediate examples of proper non-embedded minimal surfaces in $\erre^3$, such as Enneper's surface or  proper minimal surfaces with infinite topology, such the $1$-periodic Scherk's surface which in its turn satisfies inequality (\ref{hypovol}). 

We can find in the very recent literature, (see Remark 4 in \cite{Mari}), the proof that inequality (\ref{hypovol}) satisfied by the extrinsic balls $D_t$ of a minimal immersion $\varphi: P^m \to N^n$ in a Cartan-Hadamard manifold $N^n$ with sectional curvatures bounded from above $K_N \leq b \leq 0$ implies the properness of $P^m$. An older reference of the same result, but for minimal surfaces in $\erre^3$ with finite topology can be found in Theorem 3 in \cite{Qing}. 

Concerning the techniques used, the proof in \cite{Mari} is based in the intrinsic monotonicity property satisfied by the (intrinsic) volume growth $\frac{\Vol(B^N_t)}{\Vol(B^{b,m}_t)}$ of such immersions. Here, we shall denote as $B^{b,m}_t$ the geodesic $t$-ball in the real space form $\kam$, and as $S^{b,m-1}_t$ the geodesic $t$-sphere in $\kam$. On the other hand, we can find also in Lemma 4 in  \cite{Tysk} an extrinsic approach to the use of  the monotonicity formula of the extrinsic volume growth $\frac{\Vol(D_t)}{\Vol(B^{b,m}_t)}$ to prove that a minimal hypersurface $\varphi: P^{n-1} \to \erre^n$ is proper when the extrinsic volume growth is finite, (a more restrictive hypothesis than our condition (\ref{hypovol})). 

This approach  inspires the interplay between the intrinsic and the extrinsic distances used in \cite{Mari} to prove the properness of a minimal immersion in a Cartan-Hadamard manifold as well as our own way to prove Theorem \ref{Mainth0}, where we obtain a characterization of properness for non-minimal immersions in a broader family of ambient  Riemannian manifolds, namely, the Riemannian manifolds with {\em bounded geometry}. 

Recall that a complete and non-compact  Riemannian manifold $N$ is a manifold with {\em bounded geometry} if its sectional curvatures $K_N$ are bounded from above, $K_N \leq b$, ($b \in \erre^+$), and its injectivity radius $i_0(N)=\Inf\{ i_N(p) : p \in N\}$ is positive $i_0(N) =i_0 >0$. We must remark that Cartan-Hadamard manifolds are simply connected Riemannian manifolds with bounded geometry, being $b=0$ and $i_0(N)= \infty$ in this case.

In fact, we have used a comparison for the volume of geodesic balls in a submanifold  with controlled $L^p$ norm of its mean curvature immersed in such ambient manifolds, stablished by M.P. Cavalcante, H. Mirandola and F. Vitorio in \cite{Cavalcante} as a part of the proof of their Theorem B. We follow a similar vein of reasoning that these authors, but for different purposes: while they prove that the volume of the ends of the submanifold is infinite, we characterize the properness of the immersion (see Remarks \ref{infinite1} and \ref{infinite2} where we try to relate both facts).

In this regard, we have obtained the following theorem for submanifolds with bounded $L^p$-norm of their mean curvature vector, immersed in a Riemannian manifold with bounded geometry.

\begin{theorem}\label{Mainth0}
Let $\varphi:  P^m \longrightarrow N^n$ be an isometric immersion of a complete non-compact manifold $P^m$ in a manifold $N$ with bounded geometry. Assume that the mean curvature vector of $\varphi$ satisfies $\Vert \vec H_P\Vert_{L^p(P)} <\infty$, for one $m \leq p \leq \infty$.
\medskip

Then, $P^m$ is properly immersed in $N^n$ if and only if  \,$\Vol(\varphi^{-1}(B^N_t)) < \infty$ for all $t>0$, where $B^N_t$ denotes any geodesic ball of radius $t$ in the ambient manifold $N$.

\end{theorem} 

\begin{remark}\label{infinite1}
Applying Theorem B of \cite{Cavalcante} to an isometric immersion $\varphi:  P^m \longrightarrow N^n$
as in Theorem \ref{Mainth0} we have that $\Vol(P)= \infty$. In Remark \ref{infinite2} we shall extend this comment to the ends of $P$, as in Theorem B in \cite{Cavalcante}.
\end{remark}

As a corollary, we have immediately that CMC surfaces are properly immersed when inequality (\ref{hypovol}) is satisfied:

\begin{corollary}
Let $\varphi:  P^m \longrightarrow N^n$ be an isometric immersion of a complete non-compact manifold $P^m$ in a manifold $N$ with bounded geometry. Assume that the mean curvature vector of $\varphi$ has constant length.
\medskip

Then, $P^m$ is properly immersed in $N^n$ if and only if, for all $t>0$, the extrinsic ball $\varphi^{-1}(B^N_t)$ of radius $t$ has finite volume  $\Vol(\varphi^{-1}(B^N_t)) < \infty$.

\end{corollary}




A more classical approach to the study of the relationship between geometric invariants and properness of an immersion $\varphi: P^m \rightarrow N^n$ encompasses bounds on the total extrinsic curvature $\int_P \Vert B^P\Vert^2 d\sigma$, where $B^P$ denotes the second fundamental form of the immersion. Note that when we consider minimal immersions of dimension two in the Euclidean space $\erre^n$, this total extrinsic curvature agrees with total absolute curvature $\int_P \vert K_P \vert d\sigma$, where $K_P$ denotes the Gaussian curvature of the surface $P$.

Concerning this relation among the total (extrinsic) curvature of an immersion and its properness, S. Muller and V. Sverak proved in \cite{MS} that a complete and oriented surface $\varphi: M^2 \rightarrow \erre^n$ with finite total extrinsic curvature $\int_M \Vert B^M\Vert^2 d\sigma < \infty$ is proper. 

Within this study of the behavior at infinity of complete and minimal immersions with finite total extrinsic curvature, it was also proven  by M. T. Anderson in \cite{anderson} and by G. de Oliveira  in \cite{O} that the immersion of a complete and minimal submanifold $P^m$ in $\erre^n$ or $\Han$ satisfying $\int_P \Vert B^P\Vert^m d\sigma <\infty$ is proper and that $P^m$ has finite topology.

At this point we should mention the results by G. P. Bessa et al. reported in \cite{Pac} and in \cite{Pac2}, where new conditions have been stated on the decay of the extrinsic curvature for a completely immersed submanifold $P^m$ in the Euclidean space (\cite{Pac}) and in a Cartan-Hadamard manifold (\cite{Pac2}), which guarantees the properness of the submanifold and the finiteness of its topology.

In the second of our main results we deal with the relation among total curvature and properness, proving that finiteness of total curvature is not a necessary condition for properness. In order to state the theorem, we recall that given a compact subset $D\subset M$ of a Riemannian manifold $M$, an \emph{end} $E$ $M$ with respect to $D$ is a connected unbounded component of $M\setminus D$. From now on, when we say that $E$ is an end, it is implicitly assumed that $E$ is an end with respect to some compact subset $D\subset M$ (see \cite{Li2000} and \cite{DoCarmo1}).

\begin{theorem}\label{Mainth1}
Let $\varphi: M^2 \longrightarrow N^n$ be a complete Riemannian $2$-manifold immersed in a complete Riemannian manifold $N^n$. Suppose that $M$ has finite topological type, and has non-positive Gaussian curvature, $K_M\leq 0$. Let us suppose moreover that for any $t >0$,
\begin{equation}
\Vol(\varphi^{-1}(B^N_t)) < \infty.
\end{equation}
If every end $E$ of $M$ satisfies $\int_E\vert K_M \vert d\sigma = \infty$ then $M^2$ is properly immersed in $N^n$.
\end{theorem} 

\begin{remark}Observe that:

\begin{enumerate}\item This theorem can be stated on an end $E \subseteq M$ of a complete Riemannian $2$-manifold immersed in a complete Riemannian manifold $N^n$, $\varphi: M^2 \to N^n$, assuming that $K_M\leq 0$ on $E$, that $\Vol(\varphi^{-1}(B^N_t)\cap E) < \infty \,\,\forall t\, >\, 0$ and that  $\int_E\vert K_M \vert d\sigma = \infty$, and concluding that $E$ is properly immersed in $N$.

\item None of the three assumptions of this result, namely, finiteness of the volume of extrinsic balls, non-positive Gaussian curvature of the submanifold and infinite total curvature of its ends, can be avoided, as we can see in the following examples. 

In first place, we need to assume that the volume of the extrinsic balls $\varphi^{-1}(B^N_t)$ is finite for all radius $t>0$. An example of non-proper immersion with non-positive curvature, infinite total curvature and having extrinsic balls with infinite volume is Nadirashvili's minimal labyrinth $\varphi: M^2 \longrightarrow \erre^3$ (see \cite{Na}). This surface has positive first eigenvalue of the Laplacian, $\lambda_1(M) >0$, so it is hyperbolic, (see \cite{Gri}). Hence, it has infinite total curvature and infinite volume, (see \cite{I} and \cite{T}). On the other hand, we have that $\varphi(M) \subseteq B^{\erre^3}_1$, so $\Vol(\varphi^{-1}(B^{\erre^3}_1))=\Vol(M)=\infty$.

Other example of surfaces which do not satisfy the hypothesis on the volume of the extrinsic balls in our theorem are the helicoidal-Scherk examples constructed in \cite{RT}. They are simply connected minimal surfaces $\varphi: M_{n,h}^2 \to  \Hatwo \times \erre$ embedded in $\Hatwo \times \erre$ which are not properly embedded. They are invariant under discrete subgroups of isometries $G$ of $\Hatwo \times \erre$ generated by translations and screw motions, and if we consider  the quotient of $M_{n,h}^2$ by a subgroup of $G$, these surfaces have finite total curvature. However, the surfaces $M_{n,h}^2$ have infinite total curvature. They also have non-positive Gaussian curvature. These surfaces accumulates to a subset of a cylinder in $\Hatwo \times \erre$, so we have, fixing a pole $o \in \Hatwo \times \erre$ and using the argument given in the proof of Theorem \ref{Mainth0} (recall that  $M_{n,h}$ is minimal), a sequence of points diverging in $M_{n,h}$ and converging to a limit point in $\Hatwo \times \erre$. We can consider now a sequence of (disjoint) geodesic balls centered in these points. It is easy to see that the union $B$ of all these geodesic balls is included in some extrinsic ball $\varphi^{-1}(B^{\Hatwo\times\erre}_{R_0})$. We apply Lemma \ref{volu} to conclude that $\Vol(\varphi^{-1}(B^{\Hatwo\times\erre}_{R_0})) \geq \Vol(B) = \infty$.

Secondly, all the ends of the submanifold must have non-finite  total absolute curvature. Let us consider a complete and non-compact Riemannian manifold $M$ with dimension two, with constant Gaussian curvature $-1$ and finite volume $\Vol(M) < \infty$. As we mentioned above, a surface like this can be isometrically immersed (via Nash immersion theorem) in a metric ball $B^N_1$ in $\erre^N$ for large $N$. Obviously, is not properly immersed, it has negative curvature $-1$ and if we consider any end $E \subseteq M$, we have that $\int_E \vert K_M\vert d\sigma \leq \int_E \vert K_M \vert d\sigma =\Vol(M) <\infty$.

Finally, and concerning the assumption of non-positiveness of the Gaussian curvature, we can consider the computational example given by a modified pseudoesphere $S^2$, isometrically immersed in $\erre^3$, given by the following parametrization 
\begin{equation}\label{psedoesfere}
\left\{\begin{aligned}
x=&2\,\big(1+{\rm sech}\left(u\right)\cos\left(v\right)\big)\cdot \cos\left(u-\tanh\left(u\right)\right)\\
y=&2\,\big(1+{\rm sech}(u)\cos(v)\big)\cdot\sin\left(u-\tanh\left(u\right)\right)\\
z=&2\,{\rm sech}(u)\sin(v)
\end{aligned}\right.
\end{equation}
for $(u,v)\in [0,\infty)\times[0,2\pi)$. The two-dimensional volume of the extrinsic balls of $S^2$ is finite because, as can be showed numerically, the total two-dimensional volume $\Vol(S^2)$ of $S^2$ is finite, $\Vol(S^2)\approx 25.509$. Moreover, its total curvature is infinite. This immersion is not proper because $S^2$ accumulates to a circle (see figure \ref{figura}), but we must notice that $S^2$ has points where its Gaussian curvature is positive.
\end{enumerate}
\end{remark}

\begin{figure}
\begin{center}
\includegraphics[width=0.30\textwidth]{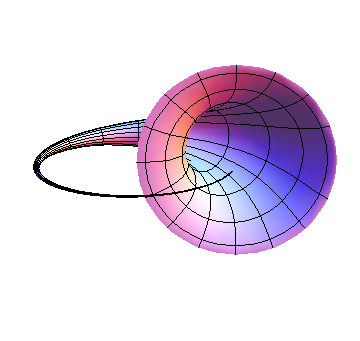}\quad\includegraphics[width=0.32\textwidth]{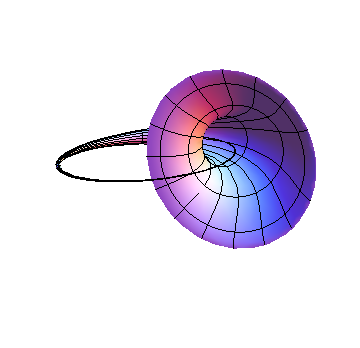}
\end{center}
\caption{Modified pseudoesphere given by parametrization (\ref{psedoesfere}). In the left side $u\in [0,10]$, and in the right $u\in[0,15]$. This surface has finite total area, and hence, finite area of its extrinsic balls, and it has infinite total curvature, $\int_{S^2}\vert K\vert d\sigma=\infty$. But this immersion is not proper because the image accumulates in the circle of radius $2$ in the plane $\{(x,y,z)\in\erre^3\,:\,z=0\}$. This example shows that the hypothesis $K_M\leq 0$ is needed in order to state theorem \ref{Mainth1}.}\label{figura}
\end{figure}

To prove Theorem \ref{Mainth1} we use Lemma \ref{MeeksRosenberg}, which is an intrinsic version 
of the proof, given in \cite{Mer}, that the injectivity radius $\I(M)$ of a complete  embedded minimal surface $M$ in $\erre^3$ of finite topology is positive. In this paper the authors proved that, if we assume that $\I(M)=0$, then there exists an end of the surface $E \subseteq M$ that has finite absolute total curvature.

A corollary of Theorem \ref{Mainth1} and the fact that every $2$-dimensional manifold with finite total absolute curvature is parabolic (see \cite{I}) is following:

\begin{corollary}\label{ParabCor}
Let $\varphi: M^2 \longrightarrow N^n$ be a complete Riemannian $2$-manifold immersed in a complete Riemannian manifold $N^n$. Suppose that $M$ has finite topological type, non-positive Gaussian curvature, $K_M\leq 0$, and that all its ends are hyperbolic. Let us suppose moreover that for any $t>0$,
\begin{equation}
\Vol(\varphi^{-1}(B^N_t)) < \infty.
\end{equation}
Then, $M^2$ is properly immersed in $N^n$.
\end{corollary}

We have adapted, on the other hand, the techniques used in Lemma \ref{MeeksRosenberg}, to relate the curvature and the topology of a complete and non-compact $2$-Riemannian manifold, using of the notion of {\em geodesic ray} and the study of the focal points of the transverse Jacobi fields to a submanifold. In particular, we have explored the relation between the minimal focal distance of a geodesic ray and the total curvature of an end containing that geodesic ray.

Before to state our last main result, we present some definitions. Recall that, given a Riemannian manifold $(M^n,g)$, a geodesic $\gamma:[0,\infty)\to M$ emanating from $p\in M$ parametrized by arc-length is called a \emph{ray} emanating from $p$ if 
$$\text{dist}_M\left(\gamma(t),\gamma(s)\right)=\vert t- s\vert$$ for all $t,s\geq 0$ (see for instance \cite{Sakai}). Whenever one has a geodesic ray $\gamma$ in a surface $M$, one can construct, for any $t>0$, the subset $\Omega_\gamma^t$  of the normal bundle $T\gamma^{\bot}$ of $\gamma$ in $M$ defined by
\begin{equation}
\Omega_\gamma^t:=\left\{(p,\xi)\in T\gamma^\bot\, :\, \vert \xi \vert <t \right\}.
\end{equation}
Denoting by $\exp^\bot$ the restriction of the exponential map to the normal bundle one can define the \emph{minimal focal distance} $\Minf(\gamma)$ of the geodesic ray $\gamma$ as (see also \S \ref{Prelim} and \cite{Tubes}) the maximum $t>0$ such that the map 
\begin{equation}
\exp^\bot : \Omega_\gamma^t\to \exp^\bot(\Omega_\gamma^t)\subset M 
\end{equation} 
is a diffeomorphism.

 Finally, we say that a geodesic ray $\gamma$ emanating from $p\in M$ belongs to the end $E$ (or equivalently, that $E$ contains the ray $\gamma$), when $\gamma(t)\in E$ for any $t$ large enough. At this point, we recall that an end can be also defined as an equivalent class of cofinal rays, (two rays in $M$, $\alpha(t)$ and $\beta(t)$ are {\em cofinal} iff for every compact set $K \subseteq  M$, there exists $t_0$ such that if $t_1, t_2 \geq t_0$, then $\alpha(t_1), \beta(t_2)$ lie in the same connected component of  $ M - K$). 

We have obtained the following intrinsic result:

\begin{theorem}\label{Mainth2}
Let $M^2$  be a complete, non compact and orientable surface with finite topology and non-positive Gaussian curvature, $K_M\leq 0$ . Suppose that there exist a geodesic ray  $\gamma \subseteq M$  with zero minimal focal distance,  $\Minf(\gamma)=0$. Then $\I(M) =0$ and, moreover,  $\gamma$ belongs to  an end $E_\gamma \subset M$  of finite total curvature, $\int_{E_\gamma} \vert K_M \vert d\sigma< \infty$.
\end{theorem}

We shall finish this introduction with an example:

\begin{example}\label{warped}

\medskip

When $M$ has negative curvature, we can use \cite[Corollary 8.6]{Tubes} to see that if $\Vol(M) <\infty$, then every ray $\gamma$ has $\Minf(\gamma)=0$ (and hence by using theorem \ref{Mainth2}, every end of $M$ has finite total curvature). However, there exist manifolds with negative curvature satisfying that every ray $\gamma$ has $\Minf(\gamma)=0$ and with $\Vol(M) =\infty$, as it is easy to check considering the $2$-Riemannian manifold $\Sigma=\erre\times_w\mathbb{S}^1$  with warping metric $g_\Sigma=dt^2+w^2(t)d\theta^2=dt^2+\text{sinh}\left(-t+\sqrt{t^2+1}\right)^2d\theta^2$. 

Observe that $\Sigma$ has finite topological type. Indeed, $\Sigma$ has the topological type of $\erre\times\mathbb{S}^1$ and, therefore $\Sigma$ has zero genus and two ends. Given $\epsilon>0$, let us consider the compact domain $\Omega_\epsilon\subset \Sigma$  given by $\Omega_\epsilon=[-\epsilon,\epsilon ]\times \mathbb{S}^1$. $\Sigma$ has, hence, two ends with respect to $\Omega_\epsilon$, namely:
\begin{equation}
\begin{aligned}
E_{-}&:=(-\infty,-\epsilon)\times \mathbb{S}^1\\
E_{+}&:=(\epsilon, +\infty)\times \mathbb{S}^1.
\end{aligned}
\end{equation}
Both of the above ends are of infinite area $2\pi \int_{\epsilon}^{\infty}\text{sinh}\left(-t+\sqrt{t^2+1}\right)dt=\infty$, and, since 
\begin{equation}
\lim_{t\to \infty} \text{sinh}\left(-t+\sqrt{t^2+1}\right)=0,
\end{equation}
the end $E_{+}$ is an shrinking end and hence every line in $\Sigma$ has zero minimal focal distance.
In fact, $E_{+}$ has finite total curvature, as it can be explicitly checked:
\begin{equation}
\begin{aligned}
\int_{E_{+}}\vert K_\Sigma\vert dV&=\int_{E_{+}} \omega''(t) dt d\theta\\= 2\pi& \int_{\epsilon}^{\infty} \left(\left(y'(t)\right)^2\text{sinh}(y(t))+y''(t)\text{cosh}(y(t))\right)dt
\approx 1.54308\\ &\text{when $\epsilon$ goes to zero}.
\end{aligned}
\end{equation}

\end{example}

\subsection{Outline}
In section \S .2 we present some basic tools, based in the Jacobi index theory, we use: first, a key comparison result for the Laplacian of the extrinsic distance in a submanifold. Secondly, and to define the notion of {\em tube} around a submanifold immersed in a Riemannian manifold, we present the transverse Jacobi fields to a submanifold and the normal exponential map in relation of the minimal focal distance of this submanifold. In Section \S.3 we state and prove Theorem \ref{Mainth0}. In Section \S.4, we prove Theorem \ref{Mainth1}. Finally, we shall prove Theorem \ref{Mainth2} in Section \S.5. In Section \S.6 we have studied some consequences of Theorem \ref{Mainth2}.

\subsection{Acknowledgments}

We would like to thank useful conversations with A. Alarc\'on, A. Hurtado, L. Jorge, F. L\'opez, and J. P\'erez. 


\section{Preliminaries}\label{Prelim}
\subsection{Extrinsic distance balls}\label{comp}
Throughout the paper we assume that $\varphi: P^m \rightarrow N^n$ is an isometric immersion of a complete non-compact Riemannian $m$-manifold $P^m$ into a complete Riemannian manifold $N^n$ with bounded geometry. We shall refer $P$ as a {\em submanifold} of $N$. 

Given a point $o \in N$, for every $x \in N^{n}-Cut \{o\}$ we
define $r_o(x) = \dist_{N}(o, x)$, and this
distance is realized by the length of a unique
geodesic from $o$ to $x$, which is the {\it
radial geodesic from $o$}. 

When $\varphi^{-1}(N-Cut \{o\}) \neq \emptyset$, we also denote by $r$
the composition
 $$r_o\circ \varphi:\varphi^{-1}(N-Cut \{o\})\subseteq  P\to \erre_{+} \cup
\{0\}$$ 

\noindent This composition is called the
{\em{extrinsic distance function}} from $o$ in
$P^m$. The gradients of $r_o$ in $N$ and $r$ in  $P$ are
denoted by $\nabla^N r_o$ and $\nabla^P r$,
respectively. Then we have
the following basic relation, by virtue of the identification, given any point $x\in P$, between the tangent vectors $X \in T_xP$ and $\varphi_{*_{x}}(X) \in T_{\varphi(x)}N$
\begin{equation}\label{radiality}
\nabla^N r_o = \nabla^P r +(\nabla^N r_o)^\bot 
\end{equation}
where $(\nabla^N r_o)^\bot(\varphi(x))$ is perpendicular to
$T_{x}P$ for all $x\in P$.

\begin{definition}\label{ExtBall}
Given $\varphi: P^m \longrightarrow N^n$ an isometric immersion of a complete and connected Riemannian $m$-manifold $P^m$ into a complete Riemannian manifold $N^n$ with bounded geometry, we denote the {\em extrinsic metric balls} of radius $t >0$, ($t \leq min\{ i_0(N), \frac{\pi}{\sqrt{b}}\}$) and center $o \in N$ by
$D_t(o)$. They are defined as  the subset of $P$:
$$
D_t(o):=\Big\{x\in P\, :\, r\left(x\right)< t\Big\}=\varphi^{-1}\left(B^N_{t}(o)\right)
$$
where $B^N_{t}(o)$ denotes the open geodesic ball
of radius $t$ centered at the point $o$ in
$N^{n}$. Note that the set $\varphi^{-1}(o)$ can be the empty set.

The extrinsic balls are the sublevel sets of the extrinsic distance function.
\end{definition}


\subsection{Transverse Jacobi fields and the normal exponential map to a submanifold}
Let $\varphi: P^m \rightarrow N^n$ be an isometric embedding of a complete non-compact Riemannian $m$-manifold $P^m$ in a Riemannian manifold $N$. Let us identify $P$ with the zero section $O(TP^\bot)$ of the normal bundle $TP^\bot$ of $P$ in $N$. Then, (cfr. \cite{Sakai}, \cite{Tubes}), we have that $\exp_{TP^\bot_{*(p,0)}} : T_{(p,0)}TP^\bot \rightarrow T_pN$ is an isomorphism for each $p \in P$, because for all $p \in P$,  we have the decomposition $T_pN=T_pP \oplus T_pP^\bot=T_{(p,0)}TP^\bot$. Then, applying the inverse mapping theorem, we have that $\exp_{TP^\bot} : TP^\bot \rightarrow N$ is a diffeomorphism if it is restricted to an open neighborhood $\Omega_P$ of $P$ in $TP^\bot$. We also shall denote $\exp_{TP^\bot}$ as $\exp^{\bot}$. Note that for any point $q \in \exp_{TP^\bot}(\Omega_P)$ there exists a unique minimal geodesic 
$\gamma$ parametrized by arc length from a point of $P$ to $q$ which realizes the distance $\dist_N(q,P)$. This geodesic $\gamma$ is perpendicular to $P$.

\begin{definition} (See \cite{Tubes}).
Let $\gamma_{\xi}(t)$ in $N$ a geodesic normal to $P$ at $t=0$. The point $q=\gamma_{\xi}(t_0)$ is a {\em cut-focal point} along $\gamma_{\xi}$ provided the distance from $q$ to $P$ is no longer minimized along $\gamma_{\xi}$ after $q$. In other words, beyond the point $q$, there is a point in the geodesic $\tilde q \in \gamma_{\xi}$ such that the distance of $\tilde q$ to $P$ is less than the distance of $q$ to $P$, and it is realized by another geodesic $\tilde \gamma_{\xi}$ that contains $\tilde q$ and meets $P$ orthogonally. 
\end{definition}

\begin{definition}
Given the unitary normal bundle $\mathbb{S}TP^\bot$ to $P$ in $N$, \emph{i.e.},
$$
\mathbb{S}TP^\bot:=\left\{(p, \xi)\in TP^\bot \, : \,  \Vert \xi\Vert =1\right\}.
$$
One has the following function $e_c : \mathbb{S}TP^\bot \to \erre$ defined as 
$$
e_c(p,\xi):= \sup_{t >0}\left\{t \, :\, \text{dist}_N(\exp^\bot(p,t\xi), P)=t\right\}.
$$
\end{definition}

\begin{remark}[See \cite{Tubes}]\label{difeo}The set $\Omega_P$ where $\exp^\bot$ is a diffeomorphism can be written as
$$
\Omega_P=\left\{(p,t\xi) \in \mathbb{S}TP^\bot\, :\, 0\leq t < e_c(p,\xi)\right\}.
$$ 
\end{remark}

\begin{definition}[See \cite{Tubes}]The {\em minimal focal} distance of $P$ in $N$ is defined as
$$ 
\Minf(P):= \inf_{(p,\xi) \in \mathbb{S}TP^\bot} \left\{e_c(p,\xi)\right\}.
$$
\end{definition}
\begin{remark}
When $P$  is a point $p$, the minimal focal distance coincides with the injectivity radius at $p$.
\end{remark}
\begin{definition}
Given $\varphi: P^m \longrightarrow N^n$ a submanifold of a Riemannian manifold $N$, and a number  $t\leq \Minf(P)$. The {\em tube of radius} $t$ about $P$ in $N$ is the set
\begin{equation*}
  T_t(P):=\left\{ x \in N\, :\, \text{dist}_N(x, P) \leq t \right\}.
\end{equation*}
\end{definition}
\begin{remark}
Since $t \leq \Minf(P)$ in the above definition and by virtue of Remark \ref{difeo}, 
\begin{equation*}
T_t(P)= \left\{ x \in N\, :\, \text{{\small $ \exists$ a geodesic $\gamma$ starting at $x$, length $L(\gamma) \leq t$ meeting $P$ orthogonally}}\right\}
\end{equation*}
\noindent and moreover, 
\begin{equation*}
T_t(P)=\exp_{TP^\bot}\Big(\left\{(p, \xi)\in TP^\bot, \text{ with }\, \Vert \xi\Vert \leq t\right\}\Big) \subseteq  \exp_{TP^\bot}\left(\Omega_P\right).
\end{equation*}

\end{remark}

\begin{definition}[See \cite{Sakai,Tubes}] Let $\gamma_{\xi}(t)$ be a geodesic in $N$ normal to $P$ at $t=0$. The point $q=\gamma_{\xi}(t_0)$ is a {\em focal point} along $\gamma_{\xi}$ if there exists a nonzero $P$-Jacobi field $Y(t)$ along $\gamma_{\xi}$ such that $Y(t_0)=0$. We call $t_0$ its {\em focal value}. Given $p \in P$ and $\xi \in \mathbb{S}TP^\bot$, let us denote as $e_f(p, \xi)$ the first of these focal values along the normal geodesic starting at $p$.
\end{definition}

\begin{remark}[See {\cite[Lemma 4.8 in Chapter II]{Sakai}}]
The point $q=\gamma_{\xi}(t_0)$ is a focal point along $\gamma_{\xi}$ if and only if rank $\exp^\bot_{*(t_0\xi)} <m$. Then, 
$$
e_f(p, \xi) =\inf_{t >0}\left\{t\,:\, \text{rank}\left(\exp^\bot_{*(t\xi)}\right) <m\right\}.
$$
\end{remark}

\begin{remark}[See {\cite[Lemma 2.11 in Chapter III]{Sakai}}]
Given $(p,\xi) \in \mathbb{S}TP^\bot$, the relation among focal points and cut-focal points along the geodesic normal to $P$, $\gamma_{\xi}(t)$, is given by:
\begin{equation}
e_c(p,\xi) \leq e_f(p,\xi)
\end{equation}
\end{remark}
The following Proposition shows that, when $P$ is totally geodesic in $N$, and $N$ has non-positive sectional curvatures, the normal exponential map is a local diffeomorphism, as occurs with the exponential map.
\begin{proposition}\label{Hadamard}
There are not focal points along a normal geodesic to a totally geodesic submanifold $P$ in a manifold of non-positive curvature $N$.
\end{proposition}
\begin{proof}
The proof is based on the convexity of the norm of Jacobi fields. Let $\gamma_{\xi}(t)$ be a geodesic  in $N$ normal to $P$ at $t=0$ (namely, $\gamma_{\xi}(0)=p \in P$ and $\gamma_{\xi}'(0)=\xi \in T_pP^\bot$). Let us suppose that $q=\gamma_{\xi}(t_0)$ is a focal point along $\gamma_{\xi}$, \emph{i.e.}, there exist a nonzero $P$-Jacobi field $Y(t)$ along $\gamma_{\xi}$ such that $Y(t_0)=0$. Then, $Y(0) \in T_pP$ and $Y'(0)-A_{\xi}(Y(0))\in T_pP^\bot$, being $A_{\xi}$ the Weingarten map of $P$.

\noindent Let us define $F(t)=\langle Y(t), Y(t)\rangle$. Then, $F'(t)= 2\langle Y'(t), Y(t)\rangle$ and 
\begin{equation}\label{ineqF}
\begin{aligned}
F''(t)=& 2 \Vert Y'(t)\Vert^2 +2\langle Y''(t), Y(t)\rangle\\
=&2 \Vert Y'(t)\Vert^2 -2K\left(\gamma_{\xi}'(t), Y(t)\right)\,\Vert \gamma_{\xi}'(t) \wedge Y(t)\Vert^2\\
 \geq & 0,
\end{aligned}
\end{equation}
because $K\left(\gamma_{\xi}'(t), Y(t)\right) \leq 0\,\,\forall t $ \,\,and  
\begin{equation}
\begin{aligned}
\langle Y''(t), Y(t)\rangle &=-\Big\langle R\left(Y\left(t\right),\gamma_{\xi}'\left(t\right)\right)\gamma_{\xi}'(t),\, Y(t)\Big\rangle\\
&=-K\left(\gamma_{\xi}'\left(t\right), Y\left(t\right)\right)\cdot\left\Vert \gamma_{\xi}'\left(t\right) \wedge Y(t)\right\Vert^2.
\end{aligned}
\end{equation}

On the other hand, $F(0)=\Vert Y(0)\Vert^2 \geq 0$ and $F(t_0) =\Vert Y(t_0)\Vert^2 = 0$. We are going to compute $F'(0)$. Since $Y'(0)-A_{\xi}(Y(0))= Z(p)\in T_pP^\bot$ and $Y(0) \in T_pP$ ,
\begin{equation}
\begin{aligned}
F'(0)=& 2\left\langle Y'\left(0\right), Y\left(0\right)\right\rangle=2\left\langle A_{\xi}\left(Y(0)\right),\, Y(0)\right\rangle\\
=&2\left\langle B^P\left(Y(0),Y(0)\right),\, \xi\right\rangle
\end{aligned}
\end{equation}

\noindent But $P$ is totally geodesic, so $F'(0)=0$. Moreover, using inequality (\ref{ineqF}) and  since $Y'(0)\neq 0$ (because $Y(t)$ is a nonzero vector field, see \cite[Chap. 5, Corollary 2.5 and Remark 2.6]{DoCarmo2}) then $F''(0) >0$.

\noindent Hence, for $t>0$ sufficiently small, $$F(t)=F(0)+F'(0) t+ \frac{F''(0)}{2}t^2+ O(t^3) > F(0) \geq 0$$
Since $F''(t) \geq 0$, then $F'(t_2) \geq F'(t_1)$ when $t_2 \geq t_1$. Therefore $F(t) >0$ for all $t>0$.
\end{proof}

\subsection{The volume of a tube around a curve}

Let us consider a curve $\sigma \subseteq N$ as a submanifold of the ambient Riemnnian manifold $N$, using the inclusion map $i: \sigma \longrightarrow N$. We have the following result, which is proved in \cite{Tubes}, 

\begin{theorem}[See {\cite[Corollary 8.6]{Tubes}}]
Let $\sigma$ be a curve with finite length $L(\sigma)$ in a Riemannian manifold $M^n$. Then, if $K_M \leq 0$, we have, for  $0 \leq r < \Minf(\sigma)$, that
\begin{equation}
\Vol(T_r(\sigma)) \geq \frac{(\pi r^2)^{\frac{n-1}{2}}}{(\frac{1}{2}(n-1))!} L(\sigma).
\end{equation}
\end{theorem}

\begin{remark}
Note that $\frac{(\pi r^2)^{\frac{n-1}{2}}}{(\frac{1}{2}(n-1))!} L(\sigma)$ corresponds to the volume of the tube of radius $r$ around a curve with length $L(\sigma)$ in $\erre^n$.
\end{remark}

\begin{corollary}\label{voltubray}
Let $\gamma$ be a geodesic ray on a complete manifold $M$ with non-positive curvature $K_M\leq 0$. Suppose that $\Minf(\gamma)>0$, then for any $0 <r<\Minf(\gamma)$,
$$
\Vol(T_r(\gamma))=\infty.
$$
\end{corollary}
\begin{proof}
Let us consider a partition of $\gamma=\cup_{i=1}^{\infty} \gamma_i$, where $\gamma_i$ is a geodesic segment with finite length $L(\gamma_i)$. As $K_M\leq 0$, one has that, as a consequence of Hotelling's tube formula, (see also \cite[Corollary 8.6]{Tubes} and the above theorem), that for any $0 <r<\Minf(\gamma)$,
\begin{equation*}
\Vol(T_r(\gamma_i)) \geq \frac{(\pi r^2)^{\frac{n-1}{2}}}{(\frac{1}{2}(n-1))!} L(\gamma_i)\,\,\,\forall i=1,...,\infty
\end{equation*}

Then, since $T_r(\gamma)=\cup_{i=1}^{\infty}T_r(\gamma_i)$, we have
\begin{equation}
\begin{aligned}
\Vol(T_r(\gamma))= \sum_{i=1}^{\infty} \Vol(T_r(\gamma_i))\geq \frac{(\pi r^2)^{\frac{n-1}{2}}}{(\frac{1}{2}(n-1))!} \sum_{i=1}^{\infty} L(\gamma_i)=\infty
\end{aligned}
\end{equation}\end{proof}

\section{Proof of Theorem \ref{Mainth0}\label{Proof1}}

In this section we shall prove the first of the main results of this paper.

In order to help the reader, we shall present again the statement of the result:

\begin{theorem*}
Let $\varphi:  P^m \longrightarrow N^n$ be an isometric immersion of a complete non-compact manifold $P^m$ in a manifold $N$ with bounded geometry. Assume that the mean curvature vector of $\varphi$ satisfies $\Vert \vec H_P\Vert_{L^p(P)} <\infty$, for one $m \leq p \leq \infty$.

Then, $P^m$ is properly immersed in $N^n$ if and only if $\Vol(\varphi^{-1}(B^N_t)) < \infty$ for all $t>0$, where $B^N_t$ denotes any geodesic ball of radius $t$ in the ambient manifold $N$.
\end{theorem*} 
\begin{proof} 

First of all we shall need the following estimate for the volume of the geodesic balls of a submanifold with bounded mean curvature immersed in  a Riemannian manifold with bounded geometry given in the proof of Theorem B in the paper \cite{Cavalcante}:

\begin{lemma}[See Theorem B in  \cite{Cavalcante}]\label{volu}
Let $\varphi:  P^m \longrightarrow N^n$ be an isometric immersion of a complete non-compact manifold $P^m$ in a manifold $N$ with bounded geometry. Assume that the mean curvature vector of $\varphi$ satisfies $\Vert \vec H_P\Vert_{L^p(P)} <\infty$, for one $m \leq p \leq \infty$.

 Then, given a fixed $p_0 \in P$ there exists  $\rho_0$ sufficiently large and a positive constant  $K >0$ such that, for all $p \in P -B^P_{2\rho_0}(p_0)$, and for all $\rho \leq \rho_0$, we have
\begin{equation}\label{ineq1}
\Vol(B^P_{\rho}(p))\geq \frac{1}{2Km}\rho^m .
\end{equation}
\end{lemma}

Now, to prove the result, let us suppose first that  the immersion is proper. As $N$ is complete, then $\varphi^{-1}(B^N_t)$ is compact for all $t>0$ and hence, $\Vol(\varphi^{-1}(B^N_t)) < \infty$.

On the other hand, if we suppose that the immersion is not proper, then there exists a point $o \in N$ and geodesic $R_0$- ball in $N$, $\bar B^N_{R_0}(o)$, with compact closure, such that $\varphi^{-1}(\bar B^N_{R_0}(o))$ is not compact. As it is closed, then $\varphi^{-1}(\bar B^N_{R_0}(o))$ is not bounded in $P$. Therefore, let us consider $\{p_i\}_{i=1}^{\infty}$ a sequence on points in $\varphi^{-1}(\bar B^N_{R_0}(o))$, which diverges in $P$ and such that $\{\varphi(p_i) \}_{i=1}^\infty$ converges to a limit point $p \in N$. 

Given $p_0$ the first point of the sequence, and applying Lemma \ref{volu}, there exists  $\rho_0$ sufficiently large and a positive constant  $K >0$ such that, for all $p \in P -B^P_{2\rho_0}(p_0)$, and for all $\rho \leq \rho_0$, we have
\begin{equation}\label{ineq1}
\Vol(B^P_{\rho}(p))\geq \frac{1}{2Km}\rho^m .
\end{equation}

On the other hand, as the sequence $\{p_i\}_{i=1}^{\infty}$ diverges in $P$, we know that $\{p_i\}_{i=n_0}^{\infty} \subseteq P -B^P_{2\rho_0}(p_0)$  for some $n_0 \in \ene$.

Let us consider now, given the quantity $\rho_0$, a subsequence of $\{p_n\}_{n=1}^{\infty}$ formed by points $p_n$ satisfying $\dist_P(p_n,p_{n+1}) >  \rho_0 \,\,\forall n \in \ene$, (see Remark \ref{infinite2}). Now, we shall consider the geodesic balls $B^P_{\rho_0/2}(p_n)$.

With these radius, the balls $B^P_{\rho_0/2}(p_n)$ are pairwise disjoint. Moreover, we have that, by Lemma \ref{volu}, $\Vol(B^P_{\rho_0/2}(p_n)) \geq \frac{1}{2Km}(\frac{\rho_0}{2})^m$ for all $n \geq n_0$.

 Let us see that  $\cup_{n=n_0}^{\infty} B^P_{\rho_0/2}(p_n) \subseteq D_{R_0+\rho_0/2}(o)=\varphi^{-1}(B^N_{R_0+ \rho_0/2}(o))$. For that, let us consider $x \in \cup_{n=n_0}^{\infty} B^P_{\rho_0/2}(p_n)$. Then, there exists $n_1 \geq n_0$ such that  $x \in  B^P_{\rho_0/2}(p_{n_1})$, so, as $p_{n_1} \in \varphi^{-1}(\bar B^N_{R_0}(o))$, we have
\begin{equation}
\begin{aligned}
\dist_N\left(\varphi(x),o\right) \leq &\dist_N\left(\varphi(x), \varphi(p_{n_1})\right) + \dist_N(\varphi(p_{n_1}), o) \\ \leq & \dist_P\left(x, p_{n_1}\right)+ R_0\\ \leq& \rho_0/2+ R_0
\end{aligned}
\end{equation}
Then, as we know, for all $n\geq n_0$, that $\Vol\left(B^P_{\rho_0/2}(p_n )\right) \geq \frac{1}{2Km}(\frac{\rho_0}{2})^m>0$,  then

\begin{equation}
\begin{aligned}
\Vol( D_{R_0+\rho_0/2} ) &\geq \Vol\left(\bigcup_{n=n_0}^\infty B^P_{\rho_0/2}(p_n)\right) \\ & \geq \sum_{n=n_0}^{\infty} \Vol\left(B^P_{\rho_0/2}(p_n)\right) \geq \sum_{n=n_0}^{\infty} \frac{1}{2Km}(\frac{\rho_0}{2})^m =\infty
\end{aligned}
\end{equation}
\noindent which is a contradiction with the hypothesis.
\end{proof}
\begin{remark}\label{infinite2}
When we consider $\{q_i\}_{i=1}^{\infty}$ a sequence on points  which diverges in $P$, (in the sense that $\lim_{n \to \infty} d_P(q_1, q_n) =\infty$), and given a fixed point $p_0 \in P$ and  the radius $\rho_0$ as in Lemma \ref{volu} it is always possible to choose a subsequence  $\{q_k\}_{k=1}^{\infty} \subseteq P$ and one radius $\delta=d_P(q_1,q_2)$ such that $B^P_{\delta}(q_k) \cap B^P_{\delta}(q_{k'}) = \emptyset$ for all sufficiently large $k \neq k'$ and such that, for all $k \geq k_0$, $B^P_{\delta} (q_k) \subseteq P - B^P_{2\rho_0}(p_0) $ in such a way that
$$\Vol(P) \geq \Vol\left(\bigcup_{k=k_0}^\infty B^P_{\delta}(q_k)\right)  \geq \sum_{k=k_0}^{\infty} \Vol\left(B^P_{\delta}(q_k)\right) \geq \sum_{k=k_0}^{\infty} \frac{1}{2Km}\delta^m =\infty$$

In particular, given a point $p_0 \in P$  and  the radius $\rho_0$ as in Lemma \ref{volu}, if we consider, as in the paper \cite{Cavalcante}, an end $E\subseteq P$ with respect the ball $B^P_{2\rho_0}(p_0)$, we can choose a sequence of points $q_2,q_3,...$ in $P$ such that 
$$q_k \in E \cap (B^P_{2k\rho_0}(p_0) -B^P_{(2k-1)\rho_0}(p_0))$$
Then, $\lim_{k \to \infty} d_P(q_2, q_k)=\infty$. Moreover, for all $0<R \leq \rho_0$ and for all $k \geq 2$, $B^P_R(q_k) \subseteq E-B^P_{2k\rho_0}(p_0))$, because $d_P(q_k, p_0) > (2k-1)\rho_0>2\rho_0$, and $B^P_{R}(q_k) \cap B^P_{R}(q_{k'}) = \emptyset$. Hence, fixing a radius $0<R \leq \rho_0$,
\begin{equation*}
\begin{aligned}
\Vol(E) &\geq \Vol(E- B^P_{2\rho_0}(p_0)) \geq \Vol\left(\bigcup_{k=2}^\infty B^P_{R}(q_k)\right)  \\& \geq \sum_{k=2}^{\infty} \Vol\left(B^P_{R}(q_k)\right) \geq \sum_{k=2}^{\infty} \frac{1}{2Km}R^m =\infty
\end{aligned}
\end{equation*}

\end{remark}

\section{Proof of Theorem \ref{Mainth1}}\label{Proof main}

In this section we shall prove Theorem \ref{Mainth1}. In order to help the reader, and as in Section \S.3, we shall present again the statement of the results.

\subsection{A previous lemma: on the positiveness of the injectivity radius}

In the paper \cite{Mer} it is proved that the injectivity radius $\I(M)$ of a complete  embedded minimal surface $M$ in $\erre^3$ of finite topology is positive, by showing that, if we assume that $\I(M)=0$, then there exists an end of the surface $E \subseteq M$ that has finite absolute total curvature so as $E$ is complete and embedded, it is asymptotic to a half catenoid or the end of a plane and the injectivity radius of $M$ restricted to $E$ is bounded away from zero, which is a contradiction and hence $\I(M)>0$. From an intrinsic point of view, we can state this result as follows proving it with exactly the same argument that in \cite{Mer}:

\begin{lemma}\label{MeeksRosenberg}(See \cite{Mer})
Let $M$ be a complete and non-compact 2-Riemannian manifold with finite topology, non-positive curvature $K_M \leq 0$ and with zero injectivity radius $\I(M)=0$. Then, there exists and end $E \subseteq M$ such that $\int_E \vert K_M\vert d\sigma  < \infty$.
\end{lemma} 
\begin{proof}
Recall that the injectivity radius of $p\in M$ is defined as the supremum 
$$\I_M(p):= \sup\Big\{t>0\,\vert\, \exp_p : B^{T_pM}_t(0_p) \rightarrow B^M_t(p) \text{ is a diffeomorphism} \Big\}$$ or, equivalently, $\I_M(p)=\dist_M(p, C_p),$ being $C_p$ the cut-locus of $p$ in $M$. The injectivity radius of $M$ is defined as $\displaystyle\I(M):=\inf_{p\in M}\{ \I_M(p)\}$. 

If we assume that $\I(M)=0$, there exists therefore a divergent sequence of points $\{p_n\}_{i=1}^\infty$ of $M$ such that $\displaystyle\lim _{n \to \infty} \I_M(p_n)=0$. Because on the contrary, if $\displaystyle\lim_{n \to \infty} p_n=p_0 \in M$, since $\I_M: M \rightarrow \erre_+$ defined as $p \rightarrow \I_M(p)$ is continuous,  then $$\displaystyle\I_M(p_0)=\I_M(\lim_{n \to \infty} p_n)=\lim_{n \to \infty} \I_M(p_n)=0,$$ and this is in contradiction with the fact that, for all $p \in M$, there exists a small positive radius $\epsilon_p>0$ such that the exponential map $\displaystyle\exp_p\vert_{B^{T_pM}_{\epsilon_p}(0_p)}$ is a diffeomorphism onto an open set in $M$. 

On the other hand, $M^2$ has finite topological type which is equivalent to be homeomorphic to the interior of a compact surface with boundary. Each connected component of this boundary is identified with an end of $M$ and is homeomorphic in its turn to a punctured disk. We call these ends  {\em annular} ends.

Hence, there exists an annular end $E$ of $M$ such that a divergent subsequence of $\{ p_n\}_{n=1}^{\infty}$ is included in $E$.

Let $\gamma_n$ be an embedded geodesic loop based at the point $p_n$, which is smooth except perhaps at $p_n$. This loop exist since $M$ has non-positive curvature and hence there are not conjugate points aloneg $\gamma$. We know that the length of this loop is $2 \I_M(p_n)$ and that the external signed angle corresponding to the vertex $p_n$  is  $\theta_n \in (-\pi, \pi)$.  

 Now, we are going to see that $\gamma_n$ is not the boundary of a disk in $M$. If $\gamma_n$ would be the boundary of a disk in $M$, then Gauss-Bonnet theorem should imply that 
\begin{equation}
\int_{\Omega} K_M d\sigma=2\pi \chi(\Omega)-\theta_n> 0,
\end{equation}
\noindent where we have used that $\chi(\Omega)=1$ because $\Omega$ is assumed to be a disk. But taking into account that $K_M\leq 0$, the above inequality is a contradiction.

Now, let us consider a compact annulus $E(n) \subset E$ bounded by the geodesic loops $T_1$ and $T_n$. Applying global Gauss-Bonnet theorem to the domain $E(n)$, we have 

\begin{equation}
0=2\pi \chi\left(E(n)\right)=  \int_{E(n)} K_M d\sigma +\theta_1+\theta_n<\int_{E(n)} K_M d\sigma +4\pi
\end{equation}
Since the above inequality does not depend on $n$,  $E$ has finite absolute total curvature and the lemma follows.
\end{proof}

\begin{remark}\label{rem4.2}Observe that from the proof of the above lemma if we have a divergent sequence $\{p_i\}_{i=1}^\infty$ of points contained in an annular end $E\subset M$, of a surface $M$ with non-positive curvature, and the sequence is such that
$$
\lim_{n\to\infty}\I(p_n)=0.
$$
Then
$$
\int_E\vert K_M\vert\, d\sigma<\infty.
$$
\end{remark}

\subsection{Proof of Theorem \ref{Mainth1}}
\medskip

We are going to prove Theorem \ref{Mainth1}, which asserts:

\begin{theorem*}
Let $\varphi: M^2 \longrightarrow N^n$ be a complete Riemannian $2$-manifold immersed in a complete Riemannian manifold $N^n$. Suppose that $M$ has finite topological type, and has non-positive Gaussian curvature, $K_M\leq 0$. Let us suppose moreover that for any $t>0$,
\begin{equation}
\Vol(\varphi^{-1}(B^N_t)) < \infty.
\end{equation}
If every end $E$ of $M$ satisfies $\int_E\vert K_M \vert d\sigma = \infty$ then $M^2$ is properly immersed in $N^n$.
\end{theorem*} 
\begin{proof}

Let us suppose that $M$ is not properly immersed. Then, there exists a geodesic ball in $N$, $\bar B^N_{R_0}(o)$, with compact closure, such that $\varphi^{-1}(\bar B^N_{R_0}(o))$ is not compact. Since it is closed, then $\varphi^{-1}(\bar B^N_{R_0}(o))$ is not bounded in $M$.

Therefore, as in the proof of Theorem \ref{Mainth0}, let us consider $\{p_i\}_{i=1}^{\infty} \subseteq \varphi^{-1}(\bar B^N_{R_0}(o))$ a sequence on points which diverges in $M$ and such that $\{\varphi(p_i) \}_{i=1}^\infty$ converges to a limit point $p \in N$. Since the topology of $M$ is finite, then there exists an annular end $E$ such that a divergent subsequence of $\{ p_n\}_{n=1}^{\infty}$ is included in $E$. Denoting as $\{q_n\}_{n=1}^{\infty}$ such subsequence, and given a positive quantity $\epsilon_1$, let us choose a subsequence of $\{q_n\}_{n=1}^{\infty}$ formed by points $q_n$ satisfying $\dist_M(q_n,q_{n+1}) >  \epsilon_1$.

  If every end $E$ has infinite total curvature, then applying Lemma \ref{MeeksRosenberg}, we have that $\I(M)=\delta>0$. Hence, given the sequence $\{q_n\}_{n=1}^{\infty}$ formed by points $q_n$ satisfying $\dist_M(q_n,q_{n+1}) >  \epsilon_1$, we have that  $\inf\{ \inj_M(q_i)\}_{i=1}^\infty\geq \delta >0$, so we shall consider now the geodesic balls $B^M_{\epsilon}(q_i)$ with $\epsilon < \text{min}\{\delta, \frac{\epsilon_1}{2}\}$, in such a way that they are pairwise disjoint.

Now, we proceed as in the proof of Theorem \ref{Mainth0} again, showing in the same way that 
$$\displaystyle\overset{\infty}{\underset{i=1}{\cup}} B^M_{\epsilon}(q_i) \subseteq D_{R_0+\epsilon}=\varphi^{-1}(B^N_{R_0+ \epsilon}(o))$$



Then, since $K_M \leq 0$ and $\epsilon < \inj(M)$, we have, by using Bishop's-Gromov comparison theorem, that $\Vol(B^M_{\epsilon}(q_i) ) \geq \Vol(B^{\erre^n}_{\epsilon})$, thus

\begin{equation}
\begin{aligned}
\Vol\left( \varphi^{-1}(B^N_{R_0+ \epsilon}(o))\right) \geq &\Vol\left(\overset{\infty}{\underset{i=1}{\cup}}B^M_{\epsilon}(q_i)\right) \\ \geq & \sum_{i=1}^{\infty} \Vol\left( B^M_{\epsilon}(q_i)\right) \geq \sum_{i=1}^{\infty} \Vol\left(B^{\erre^n}_{\epsilon}\right) =\infty
\end{aligned}
\end{equation}

\noindent which is a contradiction with the hypothesis. Hence, $M$ is properly immersed.

\end{proof}

\section{Proof of Theorem \ref{Mainth2}}

In this section, we shall prove Theorem \ref{Mainth2}, which asserts:

\begin{theorem*}
Let $M$  be a complete, non compact and orientable surface with finite topology and non-positive Gaussian curvature, $K_M\leq 0$ . Suppose that there exist a geodesic ray  $\gamma \subseteq M$  with zero minimal focal distance,  $\Minf(\gamma)=0$. Then $\I(M) =0$ and, moreover,  $\gamma$ belongs to  an end $E_\gamma \subset M$  of finite total curvature, $\int_{E_\gamma} \vert K_M \vert d\sigma< \infty$.
\end{theorem*}
 \begin{proof}
 Since $\gamma$ is a geodesic ray  such that $\Minf(\gamma)=0$, then 
\begin{equation}
\inf_{(p,\bar u) \in \mathbb{S}T\gamma^\perp}e_c(p,\bar u)=0
\end{equation}
with $\mathbb{S}T\gamma^\perp$ being the unitary normal bundle of $\gamma$ in $M$. Hence, there exist a sequence of points  
$$
\left\{(p_n, \bar u_n)\right\}_{n=1}^{\infty} \subset \mathbb{S}T\gamma^\perp
$$ 
such that 
\begin{equation}\label{limit}
\lim_{n \to \infty} e_c(p_n,\bar u_n)=0.
\end{equation}
Since $e_c$ is continuous, $\{(p_n, \bar u_n)\}_{n=1}^{\infty}$ diverges in the unitary normal bundle $\mathbb{S}T\gamma^\perp$, because on the contrary, if $(p_o, \bar u_0)= \lim_{n \to infinity} (p_n, \bar u_n)$, then $e_c(p_0,\bar u_0)=0$ which is a contradiction with the fact that $exp^\bot$ is a diffeomorphism in an open set $\Omega_\gamma$, (see Remark \ref{difeo}).

Taking into account that $\Vert \bar u_n\Vert =1 \,\,\forall n$, then $\{ p_n\}_{n=1}^{\infty}$ diverges in $M$ and then, as the topology of $M$ is finite, there exists an annular end $E_\gamma$ such that a divergent subsequence of $\{ p_n\}_{n=1}^{\infty}$ is included in $E_\gamma$ (in fact, the ray $\gamma$ is included in $E_\gamma$). We are going to prove that the infimum of the injectivity radius of such a subsequence is zero and hence by applying lemma \ref{MeeksRosenberg} (see remark \ref{rem4.2}) $E_\gamma$ has finite total curvature.

 Let $\alpha_{\bar u_n}^{p_n}:[0,\infty)\to M$ be the geodesic curve starting at $p_n$ ($\alpha_{\bar u_n}^{p_n}(0)=p_n$) with initial velocity $\bar u_n$ ($\dot \alpha_{\bar u_n}^{p_n}(0)=\bar u_n$). Then, the points
\begin{equation}
q_n=\alpha_{\bar u_n}^{p_n}\left(e_c\left(p_n,\bar u_n\right)\right)
\end{equation}
are cut-focal points of $\gamma$, where the distance from $p_n$ is no longer minimized along $\alpha_{\bar u_n}^{p_n}$ after $q_n$. By definition, thus,
\begin{equation}
\text L_n:={\rm dist}_M\left(p_n,q_n\right)=e_c\left(p_n,\bar u_n\right).
\end{equation}
Hence, by using the limit (\ref{limit}), we have that $\displaystyle\lim_{n \to \infty} L_n=0$ so, for any $\epsilon>0$, there exists $N$ large enough such that for any $n>N$
 \begin{equation}
\text{dist}_M\left(p_n,q_n\right)<\epsilon.
\end{equation}
Therefore, there exists a divergent subsequence of cut-focal points $\{ q_n\}_{n=1}^{\infty}$ included in the same end $E_\gamma$.

Let us consider now
$$
 \Omega_{\gamma}:=\left\{(p, t \bar u)\in \mathbb{S}T\gamma^\perp,\quad 0\leq t < e_c(p,\bar u)\right\}.
$$ 
It is known (see remark \ref{difeo}) that $\exp^{\bot}: \Omega_{\gamma} \rightarrow \exp^{\bot} (\Omega_{\gamma})$ is a diffeomorphism and that the boundary $\partial \exp^{\bot} (\Omega_\gamma)$ of $\exp^{\bot} (\Omega_\gamma)$ is the set of cut-focal points of $\gamma$. Hence, there exists an infinite divergent sequence of $\{q_n\}_{n=1}^\infty$ in $E_\gamma\cap \partial \exp^{\bot} (\Omega_\gamma)$. 

For these cut-focal points $q_n$, either there is more than one minimizing geodesic joining $q_n$ and $\gamma$ or $\ker\left(\exp^{\bot}_{*q}\right) \neq \{0\}$. But this last possibility implies that $q_n$ is a focal point of $\gamma$ in $M$, and this is not possible because $\gamma$ is a totally geodesic submanifold in the negatively curved manifold $M$, see Proposition \ref{Hadamard}.

Therefore, given a point $p_n \in E_\gamma$, let us consider the cut-focal point $q_n\in E_\gamma$, being  $\alpha_{\bar u_n}^{p_n}(t)$ the geodesic which realizes the distance $L_n=e_c(p_n,\bar u_n)$  from $p_n \in \gamma$ to $q_n$.

 This geodesic meets $\gamma$ orthogonally at the point $p_n$. Since $\exp^{\bot}$ is not injective when we restrict ourselves to the set of cut-focal points, then there exists another geodesic $\beta_n$ minimizing the distance between $q_n$ and $\gamma$. This geodesic will meet $\gamma$ orthogonally at the point $r_n \in \gamma \subseteq E_\gamma$. 
 
Now, let us consider the geodesic triangle $T_n \subseteq E_\gamma$ passing through the vertices $p_n$, $q_n$ and $r_n$, formed by the union of the geodesics $\gamma$, $\alpha_{\bar u_n}^{p_n}$ and $\beta_n$, namely $T_n=\gamma \cup \alpha_{\bar u_n}^{p_n} \cup \beta_n$.

Since $\alpha_{\bar u_n}^{p_n}$ and $\beta_n$ meet $\gamma$ orthogonally, the external signed angles corresponding to the vertices $p_n$ and $r_n$ are $\pm \frac{\pi}{2}$  and the angle between $\alpha_{\bar u_n}^{p_n}$ and $\beta_n$ at the point $q_n$,  $\theta_n \in (-\pi, \pi)$.  At this point, we should note that $q_n$ is not a cusp, (namely, $\vert \theta_n \vert \neq \pi$), because in this case $\alpha_{u_n}^{p_n}$ and $\beta_n$ should be the same geodesic. 
Therefore, the sum of the signed angles is
\begin{equation}
-2\pi<\text{Sum of the external signed angles of } T_n<2\pi
\end{equation}
Observe, that in the particular case when $r_n=p_n$ we have a geodesic loop.

Let us going now to see that $T_n$ can not bound a disk. To see this, let us suppose that $\Omega \subseteq E_\gamma$ is a disk such that $\partial \Omega =T_n$. Since $T_n$ is a piecewise curve formed by geodesic segments, by applying Gauss-Bonnet theorem 
\begin{equation}\label{nodisk}
\int_{\Omega} K_M d\sigma>2\pi \chi(\Omega)-2\pi= 0,
\end{equation}
\noindent where we have used that $\chi(\Omega)=1$ because $\Omega$ is assumed to be a disk. But taking into account that $K_M\leq 0$, the above inequality is a contradiction. 

Finally, to see that ${\rm inj}(M)=0$, we are going to show that $\I_M(p_n) \leq 4 L_n$, and hence, $\displaystyle\lim _{n \to \infty} \I_M(p_n)\leq 4 \lim_{n \to \infty} L_n=0$. To do it, we shall prove first that $T_n \subseteq B^M_{4L_n}(p_n)$ and to see this, let us consider $x\in T_n$. Since $T_n=\gamma \cup \alpha_{\bar u_n}^{p_n} \cup \beta_n$, we have that $x\in \gamma$, or $x \in  \alpha_{\bar u_n}^{p_n}$ or $x \in \beta_n$. Recall that $\text{dist}_M(p_n,q_n)=\text{dist}_M(q_n,r_n)=L_n=e_c(p_n,\bar u_n)$. Then:

If $x\in \alpha_{\bar u_n}^{p_n}$, then
\begin{equation}
\text{dist}_M(p_n,x)\leq \text{dist}_M(p_n,q_n)\leq L_n
\end{equation}

If $x\in \beta_n$, then
\begin{equation}
\begin{aligned}
\text{dist}_M(p_n,x)\leq & \text{dist}_M(p_n,q_n)+\text{dist}_M(q_n,x)\\
\leq&\text{dist}_M(p_n,q_n)+\text{dist}_M(q_n,r_n)\\
\leq& 2 L_n
\end{aligned}
\end{equation}

Finally, if $x\in  \gamma$, then, as $\text{dist}_M(r_n, p_n) \leq \text{dist}_M(p_n,q_n)+\text{dist}_M(q_n,r_n)$,
\begin{equation}
\begin{aligned}
\text{dist}_M(p_n,x)\leq&\text{dist}_M(p_n,q_n)+\text{dist}_M(q_n,r_n)+\text{dist}_M(r_n,x) \\
\leq&\text{dist}_M(p_n,q_n)+\text{dist}_M(q_n,r_n)+\text{dist}_M(r_n,p_n)\leq  4 L_n
\end{aligned}
\end{equation}Hence $T_n \subseteq B^M_{4L_n}(p_n)$. If $\I_M(p_n) > 4L_n$, then $T_n$ (which can not bound a disk) is included in a domain which is homeomorphic to a disk, which is a contradiction by using the Jordan-Schoenflies theorem (see \cite{Cairns} for instance).
\end{proof}

\section{On the minimal focal distance of a geodesic ray}\label{cor-sec}

As a first corollary of  Theorem \ref{Mainth2} we have the following result, assuming that the surface has negative curvature. Namely:

 \begin{corollary}\label{Mainth6}  Let $M$  be a complete, non compact, and orientable surface  with finite topology, non-positive Gaussian curvature $K_M \leq 0$ and such that every ray $\gamma$ satisfies that $\Minf(\gamma)=0$. Then, $M$  has finite total curvature, $\int_M \vert K_M\vert d\sigma<\infty$.
\end{corollary}
\begin{proof}
The proof of Corollary \ref{Mainth6} follows directly from the hypothesis that {\em all} the rays in $M$ have zero focal distance and that $M$ has finite topology.
\end{proof}

In this section we provide some analytic and geometric sufficient and necessary  conditions for the existence of geodesic rays with zero minimal focal distance. First of all, note that as an immediate consequence of Corollary \ref{voltubray}, we have the following:
\begin{corollary} \label{Mainth7}
Let $M$ be a complete, non compact, and orientable surface with non-positive Gaussian curvature and finite area. Then, for any geodesic ray $\gamma$ in $M$ ,
\begin{equation}
\Minf(\gamma)=0.
\end{equation}
\end{corollary}
Hence, in the non-positively curved setting, finite area implies zero minimal focal distance for any geodesic ray. Observe, however that the converse is in general false, see for instance example  \ref{warped} in Section \S.1, where we have zero minimal focal distance for a geodesic ray and infinite area. Nevertheless, by using Theorem \ref{Mainth2}  we can state
\begin{corollary}
Let $M$ be a complete, non compact, and orientable surface with negative Gaussian curvature $K_M\leq b<0$ and finite topological type. Given a geodesic ray $\gamma$ in $M$,  $\Minf(\gamma)=0$ if and only if there exist an end $E_\gamma$ with finite area such that $\gamma$ belongs to $E_\gamma$.
\end{corollary}

Therefore, 

\begin{corollary}
Let $M$ be a complete, non compact, and orientable surface with negative Gaussian curvature $K_M<0$ and finite topological type. Suppose that every end of $M$ has infinite area. Suppose, moreover, that a geodesic ray $\gamma$ in $M$ has $\Minf(\gamma)=0$. Then, there exist an end $E_\gamma$ such that $\gamma$ belongs to $E_\gamma$ and a  divergent sequence of points $\{p_i \}_{i=1}^\infty\subset E_\gamma$, such that
\begin{equation}\label{flat}
\lim_{i\to\infty}K_M(p_i)=0.
\end{equation}
\end{corollary}

Whence, the conclusion of the above corollaries becomes clear. In a negatively curved surface of finite topological type the presence of a geodesic ray $\gamma$ with zero minimal focal distance implies that there exist an end containing $\gamma$ with finite area or the end is asymptotically flat (in the sense of the limit (\ref{flat})).

\begin{remark}
From Corollaries \ref{Mainth6} and \ref{Mainth7} we can conclude moreover that given $M$ a complete, non-compact and orientable surface with finite topology, finite area and non-positive curvature, then $M$ has finite total curvature.

\end{remark}

K. Ichihara proved in \cite{I} that if a surface (or an end) has finite total curvature, the surface (or the end) is a parabolic surface (or end). The same is true if the surface (or end) has finite total area.  Hence,

\begin{corollary}
Let $M$ be a negatively curved and orientable surface of finite topological type, if every end of $M$ is an hyperbolic end, then for any geodesic ray $\gamma$ in $M$,
\begin{equation}
\Minf(\gamma)>0.
\end{equation}  
\end{corollary}

Moreover, if a surface (or an end) has positive fundamental tone, the surface is hyperbolic (see \cite{Gri}). Recall that the fundamental tone $\lambda^*(M)$ of a surface $M$ is defined by
\begin{equation}\label{Ray} 
\lambda^{*}(M)=\inf\{\frac{\smallint_{M}\vert
\nabla f \vert^{2}dA}{\smallint_{M} f^{2}dA};\, f\in
 {L^{2}_{1,0}(M )\setminus\{0\}}\}
\end{equation}
\noindent where $L^{2}_{1,\,0}(M )$ is the completion  of
$C^{\infty}_{0}(M )$ with respect to the norm $ \Vert\varphi
\Vert^2=\int_{M}\varphi^{2}+\int_{M} \vert\nabla \varphi\vert^2$.

From the definition of the fundamental tone, one readily concludes that for any end $E\subset M$
\begin{equation}
\lambda^*(M)\leq \lambda^*(E).
\end{equation}

Thus,

\begin{corollary}
Let $M$ be a negatively curved and orientable surface of finite topological type, if  $M$ has positive fundamental tone $\lambda^*(M)>0$, then for any geodesic ray $\gamma$ in $M$,
\begin{equation}
\Minf(\gamma)>0.
\end{equation}  

\end{corollary}


\end{document}